\documentclass[12pt]{article}
\usepackage[top=1in, bottom=1in, left=1in, right=1.25in]{geometry}
\usepackage{amsmath}
\usepackage{enumerate}
\usepackage{fancyvrb}
\usepackage{amssymb}
\usepackage{color}
\usepackage{mathrsfs}

\usepackage{mathabx}
\usepackage{amsthm}
\usepackage[dvipsnames]{xcolor}
\usepackage{wasysym}
\usepackage{tikz}
\usepackage{subfig}
\usepackage[font=small]{caption}
\usetikzlibrary{cd}
\usetikzlibrary{decorations.markings}

\newcommand{\dynkinradius}{.1cm}
\newcommand{\dynkinstep}{1.5cm}
\newcommand{\dynkindot}[3]{\fill (\dynkinstep*#1,\dynkinstep*#2) circle (\dynkinradius)node[below]{#3};}
\newcommand{\dynkindott}[3]{\fill (\dynkinstep*#1,\dynkinstep*#2) circle (\dynkinradius)node[right]{#3};}

\newcommand{\dynkinline}[4]{\draw[thin] (\dynkinstep*#1,\dynkinstep*#2) -- (\dynkinstep*#3,\dynkinstep*#4);}
\newcommand{\dynkindots}[4]{\draw[dotted] (\dynkinstep*#1,\dynkinstep*#2) -- (\dynkinstep*#3,\dynkinstep*#4);}

\newenvironment{dynkin}{\begin{tikzpicture}[scale=0.5, decoration={markings,mark=at position 0.7 with {\arrow{>}}}]}
{\end{tikzpicture}}

\newtheorem{prop}{Proposition}
\newtheorem{thm}{Theorem}
\newtheorem{lemma}{Lemma}
\newtheorem{defn}{Definition}
\newtheorem{cor}{Corollary}

\def\C{\mathbb{C}}
\def\R{\mathbb{R}}

\def\P{\mathbb{P}}

\def\lra{\leftrightarrow}

\def\sign{\operatorname{sign}}

\title{A Bruhat Atlas on the Wonderful Compactification of
$PSO(2n)/SO(2n-1)$}
\author{Daoji Huang}
\date{}
\begin{document}
\maketitle
\begin{abstract}
    A stratified manifold
    has a Bruhat atlas on it if it can be covered with open charts such that each chart is stratified-isomorphic to an (opposite) Bruhat cell in a (usually Kac-Moody) flag manifold. In this paper, we construct an anticanonical stratification
     on the wonderful
    compactification of the symmetric space
    $PSO(2n)/SO(2n-1)$ and show that the open charts
    are isomorphic to certain (opposite) 
    Bruhat cells in the type $D_{n+1}$ flag manifold.
\end{abstract}
\section{Introduction and Statements of the Main Result}
\subsection{Wonderful Compactification of a Symmetric Space}
Let $G$ be a connected, semisimple, adjoint, linear algebraic group
over $\C$, and $\theta$ an involution on $G$. Denote by
$K=G^\theta$ its group of fixed points. 
Let $T$ be a $\theta$-stable maximal torus such that 
$T^{-\theta}=\{t\in T:\theta(t)=t^{-1}\}$ is of maximal dimension. Any such
$T$ is called a maximally split torus. Let $A$ be 
the connected component of the identity of $T^{-\theta}$.
Let $\Phi$ be the root system of  $(G,T)$ and $W$ the Weyl group. $\theta$ permutes the
roots of $\Phi$. Let $\Phi_0$ denote the closed subsystem of 
$\Phi$ fixed by $\theta$.
There exists a positive system $\Phi^+$ such that
if $\alpha\in \Phi^+$ and $\theta(\alpha)\in \Phi^+$, then $\alpha\in \Phi_0$.
Let $B$ be the Borel group defined by $\Phi^+$ and $\Delta$ the set of
simple roots. 
For any $J\subset \Delta$, let $P_J$ be the corresponding standard
parabolic subgroup of $G$. Let $P=P_{\Delta\cap \Phi_0}$ and $U$
be the unipotent radical of $P$. 

Let $X^*(T)$ denote the character lattice of $T$ and $X^*(A)$ 
the character lattice of $A$. We have a surjective
 homomorphism $X^*(T)\to X^*(A)$ induced by restriction. Let 
 $\Phi_A$ be the set of nonzero elements
 in the image of $\Phi$ in $X^*(A)$. Then 
 $\Phi_A$ is a root system in $X_{\R}^*(A)=\R\otimes X^*(A)$, whose
 Weyl group $W_A$ can be identified with $N_G(A)/Z_G(A)$.
 The set $\Delta_A$ of nonzero images of $\Delta$ in $X^*(A)$
form the simple roots of $\Phi_A$. Let $C\subset X_{\R}^*(A)$ be the 
fundamental Weyl chamber. Let $T_1=T/(T\cap K)$. The canonical
map $T\to T_1$ induces an isogeny $A\to T_1$, identifying the
character lattice $X^*(T_1)$ with $2X^*(A)\subset X^*(A)$. We identify
$\R\otimes X^*(T_1)$ with $X_{\R}^*(A)$.

\begin{thm}[\cite{DP}]
There exists a pointed $G$-variety $(X,h)$ with the following properties:
\begin{enumerate}[(a)]
\item The orbit map $g\mapsto gh$ $(g\in G)$ induces an isomorphism
$G/K\to Gh$,
\item The orbit map $t\mapsto th$ $(t\in T)$ induces a $T_1$-embedding
which is normal. Its associated fan in $X_{\R}^*(A)$ is the fan of Weyl
chambers of $W_A$.
\item Let $Y$ be the open set associated to the Weyl chamber $C$.
The morphism $(u,x)\mapsto u\cdot x$ defines an isomorphism of 
$U\times Y$ onto an affine open set of $X$.
\end{enumerate}
Furthermore, $X$ is unique up to a unique $G$-isomorphism. 
\end{thm}
$X$ is the wonderful compactification (sometimes written $\overline{G/K}$)
of the affine symmetric variety $G/K$. It is smooth, projective, with finitely many
smooth $G$-stable simple normal crossings divisors.
The closure of every 
$G$ orbit of lower dimension is a
intersection of a set of these divisors.
$G/K$ is $G$-isomorphic to the open orbit. 

\subsection{Bruhat Atlas}
\begin{defn}
Let $M$ be a variety. A \textbf{stratification} on $M$ 
is a family of locally closed subvarieties
$\{M_y^\circ\}_{y\in\mathcal{Y}}$
indexed by a ranked poset $\mathcal{Y}$ 
such that 
$M=\bigsqcup_{y\in \mathcal{Y}}M_y^\circ$ , and
each
$M_y^\circ$ has closure $M_y$ such that
and  $M_y=\bigsqcup_{y'\le y} M_{y'}^\circ$.
\end{defn}
\begin{defn}
Let $M$ be a complex smooth manifold with a stratification $\mathcal{Y}$
whose minimal strata are points. A \textbf{Bruhat atlas} on $(M,\mathcal{Y})$
is the following data:
\begin{enumerate}[(a)]
\item A Kac-Moody group $H$ with Borel subgroup $B_H$.
\item An open cover for $M$ consisting of open sets $U_f$ around the minimal
strata, i.e., $M=\bigcup_{f\in \mathcal{Y}_{\min}}U_f$.  
\item A ranked poset injection $v:\mathcal{Y}^{op}\to W_H$, whose 
image is a union $\bigcup_{f\in\mathcal{Y}_{\min}} [1,v(f)]$ of Bruhat intervals.
\item  For $f\in \mathcal{Y}_{\min}$, a stratified chart isomorphism 
$c_f:U_f \xrightarrow{\sim} X_\circ^{v(f)}$ satisfying
$c_f(U_f\cap M_y)=X_\circ^{v(f)}\cap X_{v(y)}$ for all $y\in \mathcal{Y}$, where $X^{v(f)}_\circ = B_H v(f) B_H/B_H$ and $X_{v(y)}=\overline{B_H^-{v(y)}B_H/B_H}$.
\end{enumerate}
\end{defn}
Although not phrased 
explicitly in the language
of Bruhat atlases,
the first example of Bruhat atlas was constructed by Snider  in the case of $Gr(k,n)$ 
with the positroid stratification \cite{Snider}. Knutson, Woo, and Yong
built one on $G/B$
with Richardson stratification
\cite{KWY}. 
He, Knutson, and Lu 
first introduced the definition of Bruhat atlases,
and
constructed Bruhat atlases
for the wonderful compactification of a group $G$ with a certain
anticanonical stratification \cite{HKL}. 
The modeling Kac-Moody
flag manifold is in general
neither finite nor affine type.

A Bruhat atlas 
on a stratified manifold is desirable, because 
the
stratification of (opposite)
Bruhat cells with Schubert varieties is well-understood and has
many nice properties.  If a Bruhat atlas
can be constructed, the manifold with this stratification 
automatically enjoys these properties. For example, in $X^w_o$ stratified by $X_v$, each open stratum is smooth, and
each closed stratum is normal and Cohen-Macaulay, with rational
singularities \cite{brion2007frobenius}. 

Our main theorem concerns the wonderful compactification of the
symmetric space $PSO(2n)/SO(2n-1)$
with the anticanonical stratification specified below. We construct a Bruhat
atlas on this space whose open charts
are modelled by certain Bruhat cells
in the type $D_{n+1}$ flag manifold.
\begin{thm}
Let $M=\overline{G/K}$ where
$G=PSO(2n)$ and $K=SO(2n-1)$
Then $M$ is $G$-isomorphic to the
projective space $\mathbb{P}^{2n-1}$.
Let $M$ 
bear the stratification generated by
the following set of divisors $x_1,x_2,x_1x_2+x_3x_4+\cdots+x_{2n-1}x_{2n}$. Then $M$ possesses a Bruhat
atlas consisting of charts
$X^{w_i}_\circ$ $(1\le i\le 2n)$ from the flag manifold
$SO(2n+2)/B$, where the $w_i$'s have
the following reduced words
$w_1=0(n-1)\cdots 32123\cdots (n-1)0$,
$w_2=n(n-1)\cdots32123\cdots (n-1)n$, and for all $2\le i\le n$,
 $w_{2i-1}=(n+1-i)(n-i)\cdots 12\cdots (n-1)0n(n-1)\cdots (n+2 -i)$, and 
$w_{2i} = w_{2i-1}^{-1}$, as illustrated below:
\begin{center}
  \begin{dynkin}
   \pgftransformcm{0.7}{0}{0}{0.8}{\pgfpoint{0cm}{0cm}}
    \foreach \x in {1,...,3}
    {
        \dynkindot{\x}{0}{$\scriptstyle \x$}
    }
    \draw[->,thick,red] (11,1.5) to [out=225, in=0] (10,0.5)
    to (1.5,0.5) to[out=180, in=165]  (1.5,0.9) to
    (10,0.9) to [out=0, in=235](10.6,1.55);
   \dynkindot{5}{0}{}
   \dynkindot{6}{0}{}
   \dynkindot{7}{0}{$\scriptstyle n-1$}
    \dynkindott{7.5}{.9}{$\scriptstyle0$}
    \dynkindott{7.5}{-.9}{$\scriptstyle n$}

    \dynkinline{1}{0}{2}{0}
    \dynkinline{2}{0}{3}{0}
    \dynkindots{3}{0}{4}{0}
    \dynkindots{4}{0}{5}{0}
    \dynkinline{5}{0}{6}{0}
    \dynkinline{6}{0}{7}{0}
    \dynkinline{7}{0}{7.5}{.9}
    \dynkinline{7}{0}{7.5}{-.9}
    \node[] at (7,-3) { \ \ \ \ \ \ \ \ \ \ \ \ \ \ \ \ \ \ };
  \end{dynkin}
  \qquad
   \begin{dynkin}
   \pgftransformcm{0.7}{0}{0}{0.8}{\pgfpoint{0cm}{0cm}}
   \draw[->,thick,red] (11,-1.5) to [out=135, in=0] (10,-0.5)
    to (1.5,-0.5) to[out=180, in=165]  (1.5,-0.9) to
    (10,-0.9) to [out=0, in=125](10.6,-1.55);
    \foreach \x in {1,...,3}
    {
        \dynkindot{\x}{0}{$\scriptstyle \x$}
    }
    
   \dynkindot{5}{0}{}
   \dynkindot{6}{0}{}
   \dynkindot{7}{0}{$\scriptstyle n-1$}
    \dynkindott{7.5}{.9}{$\scriptstyle0$}
    \dynkindott{7.5}{-.9}{$\scriptstyle n$}

    \dynkinline{1}{0}{2}{0}
    \dynkinline{2}{0}{3}{0}
    \dynkindots{3}{0}{4}{0}
    \dynkindots{4}{0}{5}{0}
    \dynkinline{5}{0}{6}{0}
    \dynkinline{6}{0}{7}{0}
    \dynkinline{7}{0}{7.5}{.9}
    \dynkinline{7}{0}{7.5}{-.9}
    \node[] at (7,-3) { \ \ \ \ \ \ \ \ \ \ \ \ \ \ \ \ \ \ };
  \end{dynkin}
  \qquad
   \begin{dynkin}
   \pgftransformcm{0.7}{0}{0}{0.8}{\pgfpoint{0cm}{0cm}}
    \draw[->,thick,red] (4.5,0.5) to  (1.5,0.5)
    
    to [out=180, in=165]  (1.5,0.9) to  (10,0.9) to [out=0, in=235]
    (11.3,2) to [out=45,in=125](12.7, 1.3) to [out=270, in=90](11.3,0)
    to [out=270, in=135](12.7,-1.3) to
    [out=270,in=0](11.3, -2)
    to [out=165,in=0] (10,-0.9) to (6,-0.9);
    \foreach \x in {1,2}
    {
        \dynkindot{\x}{0}{$\scriptstyle \x$}
    }
    \dynkindot{3}{0}{}
    \dynkindot{4}{0}{}
   
   \dynkindot{6}{0}{}
   \dynkindot{7}{0}{$\scriptstyle n-1$}
    \dynkindott{7.5}{.9}{$\scriptstyle0$}
    \dynkindott{7.5}{-.9}{$\scriptstyle n$}

    \dynkinline{1}{0}{2}{0}
    \dynkindots{2}{0}{3}{0}
    \dynkinline{3}{0}{4}{0}
    \dynkindots{4}{0}{5}{0}
    \dynkindots{5}{0}{6}{0}
    \dynkinline{6}{0}{7}{0}
    \dynkinline{7}{0}{7.5}{.9}
    \dynkinline{7}{0}{7.5}{-.9}
    \node[] at (7,-3) {and inverse of this word.};
  \end{dynkin}
  \end{center}

\end{thm}

\subsection{$\overline{PSO(2n)/SO(2n-1)}$}
Let $J$ be $2n\times 2n$ the matrix with ones on the antidiagonal  and
zeroes everywhere else. Realize the group $SO(2n)$ as the subgroup of 
$GL(2n)$ that preserves the quadratic form $J$. That is,
$SO(2n)=\{g\in GL(2n): g^TJg=J\}$. Let $J'$ be the matrix 
\[\left[\begin{smallmatrix}
 & & & & &1 \\
 &1 & & & &\\
  & &1 & & &\\
  & & & \ddots & & \\
  & & & & 1 & \\
  1 & & & & &
\end{smallmatrix}\right]\]
and $\theta$ the outer automorphism on $SO(2n)$ defined by conjugation
by $J'$. 
Let $G=PSO(2n)$, and $K=G^\theta$. 
Let $T$ be the maximal torus
formed by diagonal
matrices in $G$.
Explicitly,
$T=\{\text{diag}(t_1,t_2,\cdots,t_n,\frac{1}{t_n},\cdots,\frac{1}{t_2},\frac{1}{t_1}) :t_i\in \C^\times\}$. 
 $T$ is a maximally split torus with respect to $\theta$. 
 $T^{-\theta}=\{\text{diag}(t_1,\pm 1,\cdots,\pm 1,\frac{1}{t_1}) :t_1\in \C^\times\}$,
 and $A=\{\text{diag}(t_1, 1,\cdots, 1,\frac{1}{t_1}) :t_1\in \C^\times\}$.
 $T_1$ can be identified with $A$ quotiented out by the
  relation
 $\text{diag}(t_1,1,\cdots,1,\frac{1}{t_1})\sim \text{diag}(-t_1,1,\cdots,1,-\frac{1}{t_1}) $.

\begin{prop}
$\overline{G/K}$ is $G$-isomorphic to $\P^{2n-1}$.
\end{prop}
\begin{proof}
We check all  conditions
stated in Theorem 1.
Let $V=\C^{2n}$ be the standard representation of $SO(2n)$.
Let $h=[1,0,\cdots,0,-1]^T\in V$. 
The fixed subgroup of $\theta$ in $SO(2n)$
stabilizes this vector. Note that
$h$ is not an isotropic vector (it does not have length 0),
so the stabilizer is the group $SO(2n-1)$ fixing the $(2n-1)$-dimensional orthogonal complement of $h$. 
  Projectivize, we have $G$ acts on $\P(V)$.
  (Note that $SO(2n-1)$ has trivial center.) (a) follows using the image of
  $h$ in $\P(V)$ as the base point. 
  
  The $T$-orbit of $h$ is 
  $\{[t_1:0\cdots 0:-\frac{1}{t_1}]\in \P(V)\}$ whose closure is a
  $\P^1$. This gives a normal embedding of $T_1$. Thus (b) follows.
  
$U$ contains all elements of the form
\[\left[\begin{smallmatrix}
 1&b_{1,2}& b_{1,3} &\cdots & b_{1,2n-1}&-m\\
 &1 & & & &-b_{1,2n-1}\\
  & &1 & & & \vdots\\
  & & & \ddots & & -b_{1,3} \\
  & & & & 1 & -b_{1,2}\\
   & & & & &1
\end{smallmatrix}\right]\] where 
 $m=\sum_{i=2}^n b_{1,i}b_{1,2n-i+1}  $. 
 Meanwhie, $Y=\{[(c,0,\cdots,0,1)^T]:c\in \C\}$. Therefore,
 $U\times Y$ is isomorphic to an affine open set of $\P(V)$
 (i.e. the affine open set with the last coordinate being nonzero).
\end{proof}

We use the coordinatization
$[x_1:x_3:x_5:\cdots x_{2n-1}:x_{2n}:x_{2n-2}:\cdots :x_2]$ on 
$\P^{2n-1}=\overline{G/K}$. Let $f_1 = x_1$, $f_{1'}=x_2$,
$f_i=x_1x_2+x_3x_4+\cdots x_{2i-1}x_{2i}$ for all $2\le i\le n$
and $D_i=\{f_i=0\}$ be the corresponding divisors for all 
$i\in \{1,1',2,\cdots, n\}$. Then $D=D_1\cup D_1'\cup D_2\cdots \cup D_n$ 
is an anticanonical divisor on $\P^{2n-1}$. Let $\mathcal{Y}$ be 
the strata generated by $D$. 
Notice that $D_1$ is the closure of the unique codimension-1 $B$-orbit,
 $D_{1'}$ is the closure of the unique codimension-1 $B^-$-orbit,
and $D_n$ is the closure of the unique codimension-1 $G$-orbit. 
We show that the stratification $(\overline{G/K},\mathcal{Y})$ 
admits a Bruhat atlas, where the modeling Kac-Moody group $H$
is $SO(2n+2)$. To do this, we need to supply the following
data:
\begin{itemize}
\item An open over consisting of open sets around the 
minimal strata of $\mathcal{Y}$. This is exactly 
the standard cover by affine open sets
$U_i=\{x_i\neq 0\}$  of the projective space. 
\item A total number of $2n$ reduced words $w_i$ in $W_H$, each of length $2n-1$, to specify $X_\circ^{w_i}$ that corresponds
to exactly one of the affine chart $U_i$.
\item Isomorphisms $c_i:U_i\to X_\circ^{w_i}$ 
that preserves the stratifications.

\end{itemize}
\section{Coordinatization of the
Opposite Bruhat Cells on $SO(2n+2)/B$}
\subsection{Bott-Samelson Coordinates on $X^w_\circ$}
Given $w\in W_H$, the space $X_\circ^w$ is an affine space dimension $l(w)$. We coordinatize it
using a Bott-Samelson map.
For any reduced word $Q=(r_{\alpha_1},\cdots, r_{\alpha_{l(w)}})$
 of simple reflections in
$W_H$ 
such that $\prod Q=w$
we associate to it the Bott-Samelson variety
$$BS^Q:=P_{\alpha_1}\times^B \cdots
\times^B P_{\alpha_{l(w)}}/B$$ where $\times^B$
means divide by the $B$-action
on both sides. (More precisely, this action is $b\cdot(p_1,p_2)=(p_1b,b^{-1} p_2)$
for $b\in B$, $p_1\in P_1$, $p_2\in P_2$
where $P_1$ and $P_2$ are parabolics.)
The left $B$-invariant Bott-Samelson map
$$\beta_Q: BS^Q \to H/B,\ \ 
   [p_1,\cdots, p_{l(w)}]\mapsto p_1\cdots p_{l(w)}B/B$$
has image $X^w$. For each $P_{\alpha}$ where
$\alpha$ is a simple reflection, we have 
$z\mapsto e_{\alpha}(z)\widetilde{r_{\alpha} }\in P_{\alpha}$,
where  $\widetilde{r_\alpha}$ is a lift
 of $r_\alpha\in W_H$ to $H$,
$e_{\alpha}(z)=\exp(zR_{\alpha})$, 
and $R_{\alpha}$ is a basis
vector that spans the root space indexed by $\alpha$. 

Putting these together we get a parametrization

\begin{center}
$\begin{array}{clll}
\mathbb{A}^{l(w)}&\xrightarrow{m_Q} & BwB & \to  X^w_\circ \\ 
(z_1,\cdots, z_{l(w)})&\mapsto &\prod_{i=1}^{l(w)}  e_{\alpha_i}(z_i) \widetilde{r_{\alpha_i}} &  \mapsto
 \prod_{i=1}^{l(w)}  e_{\alpha_i}(z_i) \widetilde{r_{\alpha_i}} B/B
\end{array}$
\end{center}
where the composite is an isomorphism. 

We now describe explicitly the
parametrization of opposite Bruhat
cells in $SO(2n+2)/B$.
We again realize $SO(2n+2)$ with the 
$(2n+2)\times(2n+2)$
determinant 1 matrices that preserve the symmetric 
form given by the antidiagonal matrix.

 We use the following labeling of the 
 Dynkin diagram, and give
 explicit description of 
 the lift of the simple reflections
 in the Weyl group to
 $SO(2n+2)$.

 \begin{figure}[h]
     \centering
     \begin{dynkin}
    \foreach \x in {1,...,3}
    {
        \dynkindot{\x}{0}{\x}
    }
    
   \dynkindot{5}{0}{}
   \dynkindot{6}{0}{}
   \dynkindot{7}{0}{$n-1$}
    \dynkindott{7.5}{.9}{0}
    \dynkindott{7.5}{-.9}{$n$}

    \dynkinline{1}{0}{2}{0}
    \dynkinline{2}{0}{3}{0}
    \dynkindots{3}{0}{4}{0}
    \dynkindots{4}{0}{5}{0}
    \dynkinline{5}{0}{6}{0}
    \dynkinline{6}{0}{7}{0}
    \dynkinline{7}{0}{7.5}{.9}
    \dynkinline{7}{0}{7.5}{-.9}
  \end{dynkin}
    \caption{}
     \label{fig:dynkin}
 \end{figure}

 Description of the $2(n+1)\times 2(n+1)$ matrix 
 $\widetilde{r_{\alpha}}$ for $1\le \alpha \le n$:
 $$(\widetilde{r_{\alpha}})_{i,j}=\begin{cases}1 &\text{ if } i=j,\ \ i\not\in\{\alpha, \alpha+1, 2(n+1)-\alpha,2(n+1)-\alpha+1\} \\
 1 &\text{ if }(i,j)\in \{(\alpha, \alpha+1),(\alpha+1,\alpha)\}  \\
 1 &\text{ if }(i,j)\in\{(2(n+1)-\alpha, 2(n+1)-\alpha+1),(2(n+1)-\alpha+1,2(n+1)-\alpha)\} \\
 0 &\text{otherwise}
 \end{cases}$$
 
 Description of $\widetilde{r_0}$:
 $$(\widetilde{r_0})_{i,j}=\begin{cases}1 &\text{ if } i=j,\ \ i\not\in\{n,n+1,n+2,n+3\} \\
 1 &\text{ if }(i,j)\in \{(n,n+2),(n+1,n+3)\}  \\
 1 &\text{ if }(i,j)\in\{(n+2,n),(n+3,n+1)\} \\
 0 &\text{otherwise}
 \end{cases}$$
 

Explicitly when
$H=SO(2n+2)$, for $1\le \alpha \le n$, 
$e_\alpha(z)\widetilde{r_{\alpha}}$ is the identity matrix modified
by replacing the $2\times 2$ block starting at $(i,i)$ with
$\left(\begin{smallmatrix}-z & 1\\
1&0\end{smallmatrix}\right)$ and the $2\times 2$ block
starting at $(2n+2-i,2n+2-i)$ with 
$\left(\begin{smallmatrix}z & 1\\
1&0\end{smallmatrix}\right)$.  For $\alpha=0$,
$e_\alpha(z)\widetilde{r_{\alpha}}$ is the identity matrix modified by replacing
the $4\times 4$ block in the center by
$\left(\begin{smallmatrix}-z & 0 & 1 &0\\ 0 & z & 0 & 1\\
1 & 0 & 0 & 0\\ 0 & 1 & 0 & 0
\end{smallmatrix}\right)$.

For example, when $n=2$, \[ e_0(z)\widetilde{r_0}=\left[\begin{smallmatrix}
 1&0&0&0&0&0\\
 0&-z&0&1&0&0\\
 0&0&z&0&1&0\\
 0&1&0&0&0&0\\
 0&0&1&0&0&0\\
 0&0&0&0&0&1\end{smallmatrix}\right],
 e_1(z)\widetilde{r_1}=\left[\begin{smallmatrix}
 -z&1&0&0&0&0\\
 1&0&0&0&0&0\\
 0&0&1&0&0&0\\
 0&0&0&1&0&0\\
 0&0&0&0&z&1\\
 0&0&0&0&1&0\end{smallmatrix}\right],
 e_2(z)\widetilde{r_2}=\left[\begin{smallmatrix}
 1&0&0&0&0&0\\
 0&-z&1&0&0&0\\
 0&1&0&0&0&0\\
 0&0&0&z&1&0\\
 0&0&0&1&0&0\\
 0&0&0&0&0&1\end{smallmatrix}\right]\]

\subsection{Defining ideals for Schubert divisors}
Our specific choice of quadratic form for defining
$SO(2n+2)$ inside $GL(2n+2)$ ensures that
$B=B_{SO}= SO(2n+2)\cap B_{GL}$. In order to find
the defining ideals for Schubert divisors on
opposite Bruhat cells parametrized by Bott-Samelson
coordinates, we will pullback generators
for the ideals of
for matrix Schubert varieties via the parametrization.
Consider the following diagram:

\begin{center}
\begin{tikzcd}
m^{-1}(\overline{B^{GL}_- \widetilde{r_\alpha} B^{GL}}) \arrow[r] \arrow[d]
&\overline{B_- \widetilde{r_\alpha} B} \arrow[d, hook]
\arrow[r, hook] & \overline{B^{GL}_- \widetilde{r_\alpha} B^{GL}} \arrow[d,hook]\\
\mathbb{A}^{l(w)} \arrow[r,  "m" ] 
&  G \arrow[r,hook] \arrow[dr] & Mat(2n) \\
BS_\circ^Q \arrow[r, "\sim"] \arrow[u,"\simeq"]& X_\circ^w \arrow[r,hook]
\arrow[u, dotted, hook] & G/B
\end{tikzcd}
\end{center}
The image of $m$ is a lift of $X_\circ^w$ in $G$.
This is possible because $X_\circ^w$ is an affine space.
Since $X^w_\circ\cap X_{r_{\alpha}}$ is a codimension
one irreducible subvariety of $X^w_\circ$, it is defined by
one equation, given coordinates on $X^w_\circ$.
 For $x\in X^w_\circ\cap X_{r_\alpha}$,
 $\widetilde{x}\in G$ (the lift of $x$ to $G$)
 is an element of $\overline{B_-\widetilde{r_\alpha}B}\subset \overline{B_-^{GL}\widetilde{r_\alpha}B^{GL}}$.
 Now $\overline{B^{GL}_- \widetilde{r_\alpha} B^{GL}}$ is a matrix
 Schubert variety in $Mat(2n)$, whose
ideal generators are described below. 
 Pulling  back
 these generators via the parametrization $m$, we get the 
 generators of the ideal 
 for $m^{-1}(\overline{B^{GL}_- \widetilde{r_\alpha} B_{GL}})$ on $\mathbb{A}^{l(w)}$. 
 The vanishing set 
 of this ideal
 must contain $X^w_\circ \cap X_{r_\alpha}$.
 If this ideal is a principal ideal, we know that it 
 must be equal to the ideal for $X^w_\circ\cap X_{r_\alpha}$ on $X^w_\circ$. We will show that this is indeed the
 case by explicit computations in Section 4.
 
 The following definition and theorem of Fulton 
 describe a procedure to find defining ideals
 for matrix Schubert varieties.
\begin{defn}
The \textbf{essential set} $\mathcal{E}ss(\pi)$
of an $n\times n$ permutation matrix $\pi$
is defined as 
\[\mathcal{E}ss=\{(i,j)\in[1,n-1]\times[1,n-1]: \pi(i)>j, \pi^{-1}(j)>i, \pi(i+1)\le j, \pi^{-1}(j+1)\le i\}.\]
\end{defn}
Intuitively, these are the matrix entries that
are the southeast corners of the remaining 
connected regions obtained by
crossing out entries to the east and the south of
each 1 in the permutation matrix, including the 1 itself.
\begin{thm}[\cite{fulton92}]
For any permutation matrix $\pi\in S_n$ and a generic $n\times n$ matrix
$A=(x_{ij})$, 
the matrix Schubert variety $X_\pi=\overline{B_-\pi B}$
has radical ideal $I_\pi$ in the polynomial ring of
the variables $x_{11},\cdots, x_{nn}$ generated
by minors of $A$ of size $m(i,j)+1$ in the 
northwest submatrix with southeast corner $(i,j)$,
where $(i,j)\in \mathcal{E}ss(\pi)$ and $m(i,j)$
counts the number of $1$'s in that northwest submatrix
of $\pi$.
\end{thm}

\section{Wiring Diagrams}
\subsection{Computing Bott-Samelson Maps with
Wiring Diagrams}

Let $w\in W_H$ and $Q=(r_{\alpha_1},\cdots, r_{\alpha_{l(w)}})$ be a reduced word
of $w$. 
We use ``wiring diagrams'' as a combinatorial gadget
for computing the product 
$$M_Q=\prod_{i=1}^{l(w)}  e_{\alpha_i}(z_i) \widetilde{r_{\alpha_i}}.$$ 

\begin{defn}
Given a matrix $M$ of size $k\times k$,
the wiring diagram for $M$ is a weighted bipartite
graph $(U_M,V_M, E_M)$ such that
$U_M=V_M=[k]$, $E_M=\{(i,j):M_{i,j}\neq 0\}$
and for all $1\le i,j\le k$, the edge weight
$w(i,j)$ of $(i,j)$ is equal to the
matrix entry $M_{i,j}$.
Given an expression $\prod_l M_l$ 
of matrix multiplication, the wiring
diagram for this expression is the concatenation
of $M_l$'s diagrams.
\end{defn}
Given a diagram $\mathcal{D}$ for a matrix multiplication
$\prod_{l=1}^m M_l$
we let $P:i\to j$ denote the path from
the $i$th  vertex of $U_{M_1}$ and the $j$th 
vertex of $V_{M_m}$.
vertex, and $w(P)$ the product
of weights of all edges in $P$.
It is easy to see that
if $M=\prod_l M_l$, then $M_{i,j}=\sum_{P:i\to j}w(P)$.
 We say $\mathcal{D}$ computes $M$.

Using the description of the matrices
$e_\alpha(z)\widetilde{r_\alpha}$, 
in Section 2.1, given the
$Q$, we can construct the wiring
diagram for $M_Q=\prod_{i=1}^{l(w)}  e_{\alpha_i}(z_i) \widetilde{r_{\alpha_i}}$
by concatenating the diagrams for
$e_{\alpha_i}(z_i) \widetilde{r_{\alpha_i}}$.
  In all the wiring diagrams shown
  in this paper, we adopt the convention that
 a solid line means the edge has weight 1,
 a dashed line means the edge has weight
 $-z_i$,
 and a dotted line means the edge has weight
 $z_i$. We denote the diagram
 for the given $Q$ by $\mathcal{D}_Q$.
  
 The following diagram shows $\mathcal{D_Q}$ and $M_Q$ for $Q=012$ and $n=2$:
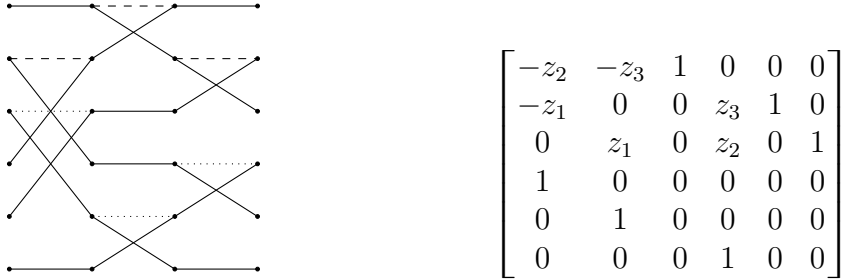
\begin{center}
   \begin{tikzpicture}
\pgftransformcm{1.1}{0}{0}{0.7}{\pgfpoint{0cm}{0cm}}
\foreach \x in {-1,0,1,2}{
      \foreach \y in {-2,-1,...,3}{
        \node[draw,circle,inner sep=0.5pt,fill] at (\x,\y) {}; } }
        \draw  (-1,2) -- (0,0) ;
\draw  (-1,0) -- (0,2) ;
\draw  [dashed](-1,2) -- (0,2) ;
\draw  (-1,-1) -- (0,1) ;
\draw  (-1,1) -- (0,-1) ;
\draw  [dotted](-1,1) -- (0,1) ;
\foreach \y in {-2,3}
{\draw (-1,\y) -- (0,\y);}
        \draw  (0,3) -- (1,2) ;
\draw  (0,2) -- (1,3) ;
\draw  [dashed](0,3) -- (1,3) ;
\draw  (0,-2) -- (1,-1) ;
\draw  (0,-1) -- (1,-2) ;
\draw  [dotted](0,-1) -- (1,-1) ;
\foreach \y in {0,1}
{\draw (0,\y) -- (1,\y);}
        \draw  (1,2) -- (2,1) ;
\draw  (1,1) -- (2,2) ;
\draw  [dashed](1,2) -- (2,2) ;
\draw  (1,-1) -- (2,0) ;
\draw  (1,0) -- (2,-1) ;
\draw  [dotted](1,0) -- (2,0) ;
\foreach \y in {-2,3}
{\draw (1,\y) -- (2,\y);}
        \node at (7,0) {$\begin{bmatrix}
         -z_2& -z_3 & 1 & 0 & 0 & 0 \\
         -z_1 & 0 & 0 & z_3 & 1 & 0 \\
         0 & z_1 & 0  & z_2 & 0  & 1 \\
         1 & 0 & 0 & 0 & 0 & 0 \\
         0 & 1 & 0 & 0 & 0  & 0 \\
         0 & 0 & 0 & 1 & 0 & 0 
        \end{bmatrix}$ };

  \end{tikzpicture}
\end{center}

Consider the diagram for a product
$M=\prod_{l=1}^m M_l$.
Let $J\subset U_{M_1},J'\subseteq V_{M_m}$ be ordered tuples of size $j$, and let $M^{(J,J')}$ denote
the $(J,J')$-minor of $M$. A matching 
$P=(P_1,\cdots, P_j):J\to J'$ of $J$ and $J'$
is a tuple of $j$ disjoint paths satisfying
the following properties:
\begin{enumerate}[(a)]
    \item There is a permutation $\sigma$
    of $\{1,\cdots, j\}$ such that
    for every $i$, $P_i$ is a path from
    $J_i$ to $J'_{\sigma(i)}$.
    \item Whenever $i\neq k$, $P_i$
    and $P_k$ have no vertices in common,
    including endpoints.
\end{enumerate}
Let $\sigma (P)$ denote the permutation
$\sigma$ from the first condition.
It counts the number of ``$\times$''
formed by this set of paths in the wiring
diagram.
The following lemma is an immediate consequence
of the Lindstr\"om-Gessel-Viennot lemma \cite{Gessel89determinants}
applied to wiring diagrams. 
\begin{lemma}
The $(J,J')$-minor of $M$ can be
computed as follows
\[M^{(J,J')}=\sum_{P=(P_1,\cdots, P_j):J\to
J'} \sign (\sigma(P))\prod_{i=1}^j w(P_i).\]
\end{lemma}

\subsection{Operations on Wiring Diagrams}
We now discuss  some operations on wiring diagrams that correspond to
operations on the corresponding matrices. 
\subsubsection{Inverse and transpose}
 First of all, 
 it is easy to see that flipping
 the wiring diagram of a matrix
 horizontally corresponds
 to taking the transpose of the matrix. 
 If $M=\prod_l M_l$ has diagram $\mathcal{D}$,
 we denote the flipped diagram as
 $\mathcal{D}_{\text{flipped}}$. $\mathcal{D}_\text{flipped}$ computes
 $M^T$.

Let $B_{\alpha,z}$ denote the diagram for
$e_\alpha(z)\widetilde{r_\alpha}$.
To get the diagram for $(e_\alpha(z)\widetilde{r_\alpha})^{-1}$ we operate on
$B_{\alpha,z}$ as follows. If $1\le \alpha\le n$, we move the horizontal edge sitting
on top of each of the two ``$\times$'' in the diagram to the bottom of the
``$\times$'', and negate the weight. This is because on the matrix level
 $\left(\begin{smallmatrix}
x & 1 \\ 1 & 0
\end{smallmatrix}\right)^{-1} = \left(\begin{smallmatrix}
0 & 1 \\ 1 & -x
\end{smallmatrix}\right)$. 
If $\alpha=0$, since $\left(\begin{smallmatrix}-x & 0 & 1 &0\\ 0 & x & 0 & 1\\
1 & 0 & 0 & 0\\ 0 & 1 & 0 & 0
\end{smallmatrix}\right)^{-1}=\left(\begin{smallmatrix}0 & 0 & 1 &0\\ 0 & 0 & 0 & 1\\
1 & 0 & x & 0\\ 0 & 1 & 0 & -x
\end{smallmatrix}\right)$, we move the horizontal edge $(n,n)$ to $(n+2,n+2)$,
$(n+1,n+1)$ to $(n+3,n+3)$, and negate their weights. Denote the new
diagram by $B_{\alpha,z}^{-1}$. The diagram that corresponds to $M_Q^{-1}$
is thus obtained by concatenating all  $B_{\alpha_i,z_i}^{-1}$ in reverse order. 
We denote this diagram by $\mathcal{D}_Q^{-1}$.
The diagram below shows $\mathcal{D}_Q^{-1}$ for $Q=012$, $n=2$.
\begin{center}
   \begin{tikzpicture}
\pgftransformcm{1.1}{0}{0}{0.7}{\pgfpoint{0cm}{0cm}}
\foreach \x in {-1,0,1,2}{
      \foreach \y in {-2,-1,...,3}{
        \node[draw,circle,inner sep=0.5pt,fill] at (\x,\y) {}; } }
        \draw  (1,2) -- (2,0) ;
\draw  (1,0) -- (2,2) ;
\draw  [dashed](1,-1) -- (2,-1) ;
\draw  (1,-1) -- (2,1) ;
\draw  (1,1) -- (2,-1) ;
\draw  [dotted](1,0) -- (2,0) ;
\foreach \y in {-2,3}
{\draw (1,\y) -- (2,\y);}

        \draw  (0,3) -- (1,2) ;
\draw  (0,2) -- (1,3) ;
\draw  [dotted](0,2) -- (1,2) ;
\draw  (0,-2) -- (1,-1) ;
\draw  (0,-1) -- (1,-2) ;
\draw  [dashed](0,-2) -- (1,-2) ;
\foreach \y in {0,1}
{\draw (0,\y) -- (1,\y);}

        \draw  (-1,2) -- (0,1) ;
\draw  (-1,1) -- (0,2) ;
\draw  [dotted](-1,1) -- (0,1) ;
\draw  (-1,-1) -- (0,0) ;
\draw  (-1,0) -- (0,-1) ;
\draw  [dashed](-1,-1) -- (0,-1) ;
\foreach \y in {-2,3}
{\draw (-1,\y) -- (0,\y);}
\end{tikzpicture}
\end{center}

\subsubsection{A duality statement}
\begin{lemma}
Let $M$ be a $k\times k$ matrix
and $J, J'\subset [k]$
such that $|J| = |J'|$. Then
\[M^{(J,J')} =\pm (M^{-T})^{([k]\setminus J,[k]\setminus J')}.\]
\end{lemma}

\begin{proof}
Let $V$ be a $k$-dimensional vector space with ordered basis
$(e_1,\cdots, e_k)$, and $M:V\to V$ be an isomorphism with
determinant 1, i.e., $\bigwedge^k M=1$. We have a
perfect pairing

$$ \bigwedge^{i} V\otimes \bigwedge^{k-i} V\to \bigwedge^k V $$
by wedging together complementary basis elements. More precisely,
 we choose the ordered basis
$(e_J)_{J\subseteq [k], |J|=i}$ 
where $e_J=e_{j_1}\wedge\cdots\wedge e_{j_i}$ for
$\bigwedge^i V$ and use the lexicographic ordering on the indices $J$.  Then we use $(e_{[k]\setminus J})_{J\subseteq [k], |J|=i}$
as the ordered basis for $\bigwedge^{k-i} V$. Hence the pairing
is simply $e_J\otimes e_{J^c}\mapsto e_J\wedge e_{J^c}$. The 
quadratic form written as a 
matrix $A$ for this pairing is a 
$\binom{k}{i}\times \binom{k}{i}$ diagonal matrix with
$\pm 1$ on the diagonal.

We then have the commuting diagram
\begin{center}
\begin{tikzcd}
\bigwedge^i V \arrow[r, "\sim"] \arrow[d, "\bigwedge^i M"] &
\bigwedge^k V\otimes (\bigwedge^{k-i} V)^* \arrow[r, "\sim"]
& \bigwedge^k V\otimes \bigwedge^{k-i} V^* \arrow[d, "\bigwedge^k M\otimes \bigwedge^{k-i}(M^T)^{-1}" ] \\
\bigwedge^i V \arrow[r,  "\sim" ] &
\bigwedge^k V\otimes (\bigwedge^{k-i} V)^*  \arrow[r, "\sim"]
& \bigwedge^k V\otimes \bigwedge^{k-i} V^*
\end{tikzcd}
\end{center}

Since $M$ has determinant 1 and $\bigwedge^k V$ is of dimension 1, the
diagram can be simplified as

\begin{center}
\begin{tikzcd}
\bigwedge^i V \arrow[r, "\sim"] \arrow[d, "\bigwedge^i M"]
&\bigwedge^{k-i} V^* \arrow[d, " \bigwedge^{k-i}(M^T)^{-1}" ] \\
\bigwedge^i V \arrow[r,  "\sim" ] 
&  \bigwedge^{k-i} V^*
\end{tikzcd}
\end{center}

On the coordinate level, the basis for $\bigwedge^{k-i} V^*$
is  $(e^*_{[k]\setminus J})_{J\subseteq [k], |J|=i}$ and 
the matrix for the top and bottom isomorphism in the diagram above is $A$. 
Again, the basis of $\bigwedge^i V$ and the (dual) basis
of $\bigwedge^{k-i} V^*$ are corresponded by complementation of indices. The
lemma then follows, as $M^{(J,J')}$
is the matrix entry in $\bigwedge^i M$ indexed
by $(J,J')$ and $(M^{-T})^{([k]\setminus J,[k]\setminus J')}$ is the matrix entry in
$\bigwedge^{k-i}(M^T)^{-1}$
indexed by $(([k]\setminus J,[k]\setminus J'))$.
\end{proof}
Lemma 2 gives the following statement
on wiring diagrams.
\begin{cor}
When $M=M_Q$ with corresponding wiring
diagram $\mathcal{D}_Q$,
the $(J,J')$-minor of $ M_Q$ obtained by 
enumerating the set of all matchings
between $J$ and $J'$ in
the diagram $\mathcal{D}_Q$ following the
formula in Lemma 1 is the same (up to sign) as the
$([2n+2]-J, [2n+2]-J')$-minor in $M_Q^{-T}$,  obtained by enumerating the set of
all matchings between
$[2n+2]-J$ and $[2n+2]-J'$ in the wiring
diagram
$(\mathcal{D}_Q^{\text{flipped}})^{-1}$.
\end{cor}

\section{Stratification Preserving Chart Isomorphisms}
In this section
we identify the opposite Bruhat cells
on $SO(2n+2)/B$ that model the charts on
$\overline{PSO(2n)/SO(2n-1)}\cong \mathbb{P}^{2n-1}$,
construct the chart isomorphisms,
and show that they are 
stratification preserving.

First we define the following distinguished words.
For ease of notation, we use
numbers $0,\cdots,n$ to denote
the simple reflections
corresponding to the simple roots,
following the labeling in
figure .
Let  
\[Q_1=0(n-1)\cdots 32123\cdots (n-1)0=:[0(n-1)\searrow 1\nearrow (n-1)0],\]
\[Q_2=n(n-1)\cdots32123\cdots (n-1)n=:[n\searrow 1\nearrow n],\] and for all $2\le i\le n$,
let 
\begin{align*}
    Q_{2i-1}&=(n+1-i)(n-i)\cdots 12\cdots (n-1)0n(n-1)\cdots (n+2 -i)\\
    & =:[(n+1-i)\searrow 1\nearrow (n-1)0n\searrow (n+2-i)],
\end{align*} and 
$Q_{2i} $ the reverse of $Q_{2i-1}$.
For all $1\le i\le 2n$, let $w_i=\prod Q_i$.

 \begin{prop}
On $X_\circ^w$ where $w=\prod Q$, $Q=[0(n-1)\searrow 1\nearrow (n-1)0]$, the 	
defining ideal for $X_{r_\alpha}\cap X_\circ^w$ is
generated by\[
\begin{cases}
z_n & \text{  if }\alpha=1 \\
z_n+z_{n-1}z_{n+1}+\cdots z_{n-(\alpha-1)}z_{n+(\alpha-1)} &\text{  if }2\le \alpha\le n\\
z_n+z_{n-1}z_{n+1}+\cdots + z_1z_{2n-1} & \text{  if }\alpha=0.
\end{cases}
\]
 \end{prop}
 \begin{proof}
 \ \\
 
\begin{tikzpicture}

\end{tikzpicture}
 
 \begin{center}
 \begin{figure}[h]
 \centering
 \subfloat[$\mathcal{D}_Q, Q=\lbrack 0(n-1) \searrow 1\nearrow (n-1)0\rbrack$]{
    \begin{tikzpicture}
\pgftransformcm{0.55}{0}{0}{0.37}{\pgfpoint{0cm}{0cm}}
\foreach \x in {-6,-5,...,7}{
      \foreach \y in {-7,-6,...,8}{
        \node[draw,circle,inner sep=0.5pt,fill] at (\x,\y) {}; } }
\draw  (-6,2) -- (-5,0) ;
\draw  (-6,0) -- (-5,2) ;
\draw  [dashed](-6,2) -- (-5,2) ;
\draw  (-6,-1) -- (-5,1) ;
\draw  (-6,1) -- (-5,-1) ;
\draw  [dotted](-6,1) -- (-5,1) ;
\foreach \y in {-6,-7,-5,-4,-3,-2,3,4,5,6,7,8}
{\draw (-6,\y) -- (-5,\y);}

\draw  (6,2) -- (7,0) ;
\draw  (6,0) -- (7,2) ;
\draw  [dashed](6,2) -- (7,2) ;
\draw  (6,-1) -- (7,1) ;
\draw  (6,1) -- (7,-1) ;
\draw  [dotted](6,1) -- (7,1) ;
\foreach \y in {-6,-7,-5,-4,-3,-2,3,4,5,6,7,8}
{\draw (6,\y) -- (7,\y);}

\draw  (-5,3) -- (-4,2) ;
\draw  (-5,2) -- (-4,3) ;
\draw  [dashed](-5,3) -- (-4,3) ;
\draw  (-5,-1) -- (-4,-2) ;
\draw  (-5,-2) -- (-4,-1) ;
\draw  [dotted](-5,-1) -- (-4,-1) ;
\foreach \y in {-6,-7,-5,-4,-3,0,1,4,5,6,7,8}
{\draw (-5,\y) -- (-4,\y);}

\draw  (5,3) -- (6,2) ;
\draw  (5,2) -- (6,3) ;
\draw  [dashed](5,3) -- (6,3) ;
\draw  (5,-1) -- (6,-2) ;
\draw  (5,-2) -- (6,-1) ;
\draw  [dotted](5,-1) -- (6,-1) ;
\foreach \y in {-6,-7,-5,-4,-3,0,1,4,5,6,7,8}
{\draw (5,\y) -- (6,\y);}

\draw  (-4,4) -- (-3,3) ;
\draw  (-4,3) -- (-3,4) ;
\draw  [dashed](-4,4) -- (-3,4) ;
\draw  (-4,-2) -- (-3,-3) ;
\draw  (-4,-3) -- (-3,-2) ;
\draw  [dotted](-4,-2) -- (-3,-2) ;
\foreach \y in {-6,-7,-5,-4,-1,0,1,2,5,6,7,8}
{\draw (-4,\y) -- (-3,\y);}

\draw  (4,4) -- (5,3) ;
\draw  (4,3) -- (5,4) ;
\draw  [dashed](4,4) -- (5,4) ;
\draw  (4,-2) -- (5,-3) ;
\draw  (4,-3) -- (5,-2) ;
\draw  [dotted](4,-2) -- (5,-2) ;
\foreach \y in {-6,-7,-5,-4,-1,0,1,2,5,6,7,8}
{\draw (4,\y) -- (5,\y);}

\draw  (-3,5) -- (-2,4) ;
\draw  (-3,4) -- (-2,5) ;
\draw  [dashed](-3,5) -- (-2,5) ;
\draw  (-3,-3) -- (-2,-4) ;
\draw  (-3,-4) -- (-2,-3) ;
\draw  [dotted](-3,-3) -- (-2,-3) ;
\foreach \y in {-6,-7,-5,-2,-1,0,1,2,3,6,7,8}
{\draw (-3,\y) -- (-2,\y);}

\draw  (3,5) -- (4,4) ;
\draw  (3,4) -- (4,5) ;
\draw  [dashed](3,5) -- (4,5) ;
\draw  (3,-3) -- (4,-4) ;
\draw  (3,-4) -- (4,-3) ;
\draw  [dotted](3,-3) -- (4,-3) ;
\foreach \y in {-6,-7,-5,-2,-1,0,1,2,3,6,7,8}
{\draw (3,\y) -- (4,\y);}

\draw  (-2,6) -- (-1,5) ;
\draw  (-2,5) -- (-1,6) ;
\draw  [dashed](-2,6) -- (-1,6) ;
\draw  (-2,-4) -- (-1,-5) ;
\draw  (-2,-5) -- (-1,-4) ;
\draw  [dotted](-2,-4) -- (-1,-4) ;
\foreach \y in {-6,-7,-3,-2,-1,0,1,2,3,4,7,8}
{\draw (-2,\y) -- (-1,\y);}

\draw  (2,6) -- (3,5) ;
\draw  (2,5) -- (3,6) ;
\draw  [dashed](2,6) -- (3,6) ;
\draw  (2,-4) -- (3,-5) ;
\draw  (2,-5) -- (3,-4) ;
\draw  [dotted](2,-4) -- (3,-4) ;
\foreach \y in {-6,-7,-3,-2,-1,0,1,2,3,4,7,8}
{\draw (2,\y) -- (3,\y);}

\draw  (-1,7) -- (0,6) ;
\draw  (-1,6) -- (0,7) ;
\draw  [dashed](-1,7) -- (0,7) ;
\draw  (-1,-5) -- (0,-6) ;
\draw  (-1,-6) -- (0,-5) ;
\draw  [dotted](-1,-5) -- (0,-5) ;
\foreach \y in {-7,-4,-3,-2,-1,0,1,2,3,4,5,8}
{\draw (-1,\y) -- (0,\y);}

\draw  (1,7) -- (2,6) ;
\draw  (1,6) -- (2,7) ;
\draw  [dashed](1,7) -- (2,7) ;
\draw  (1,-5) -- (2,-6) ;
\draw  (1,-6) -- (2,-5) ;
\draw  [dotted](1,-5) -- (2,-5) ;
\foreach \y in {-7,-4,-3,-2,-1,0,1,2,3,4,5,8}
{\draw (1,\y) -- (2,\y);}

\draw  (0,8) -- (1,7) ;
\draw  (0,7) -- (1,8) ;
\draw  [dashed](0,8) -- (1,8) ;
\draw  (0,-6) -- (1,-7) ;
\draw  (0,-7) -- (1,-6) ;
\draw  [dotted](0,-6) -- (1,-6) ;
\foreach \y in {-5,-4,-3,-2,-1,0,1,2,3,4,5,6}
{\draw (0,\y) -- (1,\y);}
  \end{tikzpicture} }
  \qquad
  \subfloat[$(\mathcal{D}_Q^{\text{flipped}})^{-1}$]{
  \begin{tikzpicture}
\pgftransformcm{0.55}{0}{0}{0.37}{\pgfpoint{0cm}{0cm}}
\foreach \x in {-6,-5,...,7}{
      \foreach \y in {-7,-6,...,8}{
        \node[draw,circle,inner sep=0.5pt,fill] at (\x,\y) {}; } }
\draw  (-6,2) -- (-5,0) ;
\draw  (-6,0) -- (-5,2) ;
\draw  [dotted](-6,0) -- (-5,0) ;
\draw  (-6,-1) -- (-5,1) ;
\draw  (-6,1) -- (-5,-1) ;
\draw  [dashed](-6,-1) -- (-5,-1) ;
\foreach \y in {-6,-7,-5,-4,-3,-2,3,4,5,6,7,8}
{\draw (-6,\y) -- (-5,\y);}

\draw  (6,2) -- (7,0) ;
\draw  (6,0) -- (7,2) ;
\draw  [dotted](6,0) -- (7,0) ;
\draw  (6,-1) -- (7,1) ;
\draw  (6,1) -- (7,-1) ;
\draw  [dashed](6,-1) -- (7,-1) ;
\foreach \y in {-6,-7,-5,-4,-3,-2,3,4,5,6,7,8}
{\draw (6,\y) -- (7,\y);}

\draw  (-5,3) -- (-4,2) ;
\draw  (-5,2) -- (-4,3) ;
\draw  [dotted](-5,2) -- (-4,2) ;
\draw  (-5,-1) -- (-4,-2) ;
\draw  (-5,-2) -- (-4,-1) ;
\draw  [dashed](-5,-2) -- (-4,-2) ;
\foreach \y in {-6,-7,-5,-4,-3,0,1,4,5,6,7,8}
{\draw (-5,\y) -- (-4,\y);}

\draw  (5,3) -- (6,2) ;
\draw  (5,2) -- (6,3) ;
\draw  [dotted](5,2) -- (6,2) ;
\draw  (5,-1) -- (6,-2) ;
\draw  (5,-2) -- (6,-1) ;
\draw  [dashed](5,-2) -- (6,-2) ;
\foreach \y in {-6,-7,-5,-4,-3,0,1,4,5,6,7,8}
{\draw (5,\y) -- (6,\y);}

\draw  (-4,4) -- (-3,3) ;
\draw  (-4,3) -- (-3,4) ;
\draw  [dotted](-4,3) -- (-3,3) ;
\draw  (-4,-2) -- (-3,-3) ;
\draw  (-4,-3) -- (-3,-2) ;
\draw  [dashed](-4,-3) -- (-3,-3) ;
\foreach \y in {-6,-7,-5,-4,-1,0,1,2,5,6,7,8}
{\draw (-4,\y) -- (-3,\y);}

\draw  (4,4) -- (5,3) ;
\draw  (4,3) -- (5,4) ;
\draw  [dotted](4,3) -- (5,3) ;
\draw  (4,-2) -- (5,-3) ;
\draw  (4,-3) -- (5,-2) ;
\draw  [dashed](4,-3) -- (5,-3) ;
\foreach \y in {-6,-7,-5,-4,-1,0,1,2,5,6,7,8}
{\draw (4,\y) -- (5,\y);}

\draw  (-3,5) -- (-2,4) ;
\draw  (-3,4) -- (-2,5) ;
\draw  [dotted](-3,4) -- (-2,4) ;
\draw  (-3,-3) -- (-2,-4) ;
\draw  (-3,-4) -- (-2,-3) ;
\draw  [dashed](-3,-4) -- (-2,-4) ;
\foreach \y in {-6,-7,-5,-2,-1,0,1,2,3,6,7,8}
{\draw (-3,\y) -- (-2,\y);}

\draw  (3,5) -- (4,4) ;
\draw  (3,4) -- (4,5) ;
\draw  [dotted](3,4) -- (4,4) ;
\draw  (3,-3) -- (4,-4) ;
\draw  (3,-4) -- (4,-3) ;
\draw  [dashed](3,-4) -- (4,-4) ;
\foreach \y in {-6,-7,-5,-2,-1,0,1,2,3,6,7,8}
{\draw (3,\y) -- (4,\y);}

\draw  (-2,6) -- (-1,5) ;
\draw  (-2,5) -- (-1,6) ;
\draw  [dotted](-2,5) -- (-1,5) ;
\draw  (-2,-4) -- (-1,-5) ;
\draw  (-2,-5) -- (-1,-4) ;
\draw  [dashed](-2,-5) -- (-1,-5) ;
\foreach \y in {-6,-7,-3,-2,-1,0,1,2,3,4,7,8}
{\draw (-2,\y) -- (-1,\y);}

\draw  (2,6) -- (3,5) ;
\draw  (2,5) -- (3,6) ;
\draw  [dotted](2,5) -- (3,5) ;
\draw  (2,-4) -- (3,-5) ;
\draw  (2,-5) -- (3,-4) ;
\draw  [dashed](2,-5) -- (3,-5) ;
\foreach \y in {-6,-7,-3,-2,-1,0,1,2,3,4,7,8}
{\draw (2,\y) -- (3,\y);}

\draw  (-1,7) -- (0,6) ;
\draw  (-1,6) -- (0,7) ;
\draw  [dotted](-1,6) -- (0,6) ;
\draw  (-1,-5) -- (0,-6) ;
\draw  (-1,-6) -- (0,-5) ;
\draw  [dashed](-1,-6) -- (0,-6) ;
\foreach \y in {-7,-4,-3,-2,-1,0,1,2,3,4,5,8}
{\draw (-1,\y) -- (0,\y);}

\draw  (1,7) -- (2,6) ;
\draw  (1,6) -- (2,7) ;
\draw  [dotted](1,6) -- (2,6) ;
\draw  (1,-5) -- (2,-6) ;
\draw  (1,-6) -- (2,-5) ;
\draw  [dashed](1,-6) -- (2,-6) ;
\foreach \y in {-7,-4,-3,-2,-1,0,1,2,3,4,5,8}
{\draw (1,\y) -- (2,\y);}

\draw  (0,8) -- (1,7) ;
\draw  (0,7) -- (1,8) ;
\draw  [dotted](0,7) -- (1,7) ;
\draw  (0,-6) -- (1,-7) ;
\draw  (0,-7) -- (1,-6) ;
\draw  [dashed](0,-7) -- (1,-7) ;
\foreach \y in {-5,-4,-3,-2,-1,0,1,2,3,4,5,6}
{\draw (0,\y) -- (1,\y);}
\end{tikzpicture}}
\caption{}
\label{wires1}
\end{figure}
\end{center}
 
 The  Fulton ideal $I_\alpha$
 of  the matrix Schubert variety
 $\overline{B^{GL}_-\widetilde{r_\alpha} B^{GL}}$ for $1\le \alpha \le n$
 is generated by the northwest $\alpha\times \alpha$
 and $(2(n+1)-\alpha)\times(2(n+1)-\alpha)$ minors. 
Pulling back these generators
via $m_Q$
to $X^w_\circ$ means computing
the northwest
$\alpha\times\alpha$ 
and $(2(n+1)-\alpha)\times(2(n+1)-\alpha)$ minors of
the matrix $M_Q=m_Q(z_1,\cdots,z_{2n-1})$. We denote the ``source''
vertices on the left  side of $\mathcal{D}_Q$ with
$\{1,2,\cdots, 2n+2\}$ and the ``target'' vertices on the right  side of
$\mathcal{D}_Q$ with $\{1',2',\cdots, (2n+2)'\}$. 
Figure 2(a) shows an example of the diagram.
 
 By Lemma 1, to compute the northwest $\alpha\times \alpha$ minor of
 $M_Q$ 
 we need to enumerate all possible ways to follow the wires
  from  $\{1,2,\cdots ,\alpha\}$
 to $\{1',2',\cdots,\alpha'\}$.
 We use the notation $j\leftrightarrow k'$ for the statement ``$j$ is 
 matched with $k'$ through a unique legal path from $j$ to $k'$.''

Notice that if $1\leftrightarrow 1'$, then $j\leftrightarrow j'$ has to be 
the case for all $2\le j\le \alpha$.
 This picks up the term $-z_n$, and there are no crossings. 
 If  $1\leftrightarrow k'$ for some $2\le k\le \alpha$, then
  $k\leftrightarrow 1'$ is forced, and $j\leftrightarrow j'$ for all $j\neq k$,
 $2\le j \le \alpha$.  
 The variables encountered are $z_{n-(k-1)}$ and $z_{n+(k-1)}$ and  
 there are $2k-3$ crossings in this matching, so
the term $-z_{n-(k-1)}z_{n+(k-1)}$ corresponds to this matching.

By Lemma 2,
 the northwest $(2(n+1)-\alpha)\times (2(n+1)-\alpha)$ minor of 
 $M_Q$ can be obtained up to sign by following the wires in the diagram
 $({\mathcal{D}_Q^{\text{flipped}}})^{-1}$ 
from  $\{2(n+1)-\alpha+1,\cdots, 2(n+1)\}$ to  $\{(2(n+1)-\alpha+1)',\cdots, 2(n+1)'\}$. 

Observe that by the structure of our diagram,
we land in the same situation as when we computed the 
$\alpha\times \alpha$ minor, only that the diagram flipped
vertically, as illustrated in Figure 2(b). Therefore the two minors may only differ by signs.
(They are in fact the same, but this detail does not matter in this
argument.) Hence, the 
ideal of $m_Q^{-1}(\overline{B^{GL}_-\widetilde{r_\alpha} B^{GL}})$ for $1\le \alpha \le n$
is generated by the
polynomial
$z_n+z_{n-1}z_{n+1}+\cdots z_{n-(\alpha-1)}z_{n+(\alpha-1)} $.

 We are left with the case when $\alpha=0$. The 
 pullback of the Fulton 
 generators given by $r_0$ are the $n\times n$ minors in the northwest
 $(n+1)\times (n+1)$ block $A$ of $M_Q$. 
From the diagram we see that the $(n+1)$th row and $(n+1)$th column of
$A$ are all 0 except the $(n+1,n+1)$-entry, which has value
$z_n+z_{n-1}z_{n+1}+\cdots + z_1z_{2n-1}$, since 
in a matching between $\{1,2,\cdots ,n+1\}$ and $\{1',2',\cdots, (n+1)'\}$, $(n+1)$ must be
matched with $(n+1)'$. We have already shown
that the northwest $n\times n$ minor is also $z_n+z_{n-1}z_{n+1}+\cdots + z_1z_{2n-1}$. From this we see that all $n\times n$ minors of $A$ are
divisible by $z_n+z_{n-1}z_{n+1}+\cdots + z_1z_{2n-1}$. Therefore the ideal
of $m_Q^{-1}(\overline{B^{GL}_-\widetilde{r_0} B^{GL}})$
is a principal ideal generated by $z_n+z_{n-1}z_{n+1}+\cdots + z_1z_{2n-1}$.
 
 \end{proof}

  \begin{prop}
On $X_\circ^w$ where $w=\prod Q$, $Q=[n\searrow 1\nearrow n]$, the 	
defining ideal for $X_{r_\alpha}\cap X_\circ^w$ is
generated by \[
\begin{cases}
z_n & \text{  if }\alpha=1 \\
z_n+z_{n-1}z_{n+1}+\cdots z_{n-(\alpha-1)}z_{n+(\alpha-1)} &\text{  if }2\le \alpha\le n\\
1 & \text{  if }\alpha=0.
\end{cases}
\]
 \end{prop}
 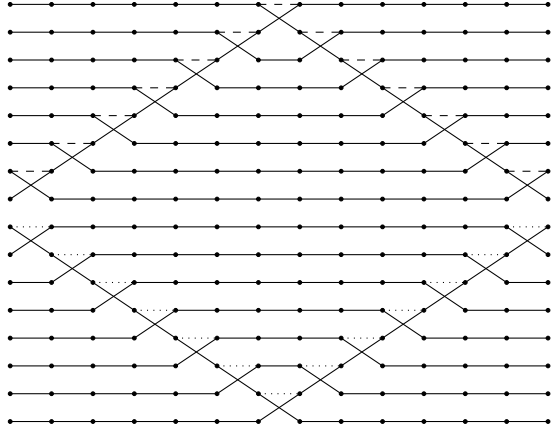
\begin{figure}[h]
     \centering
     
 \begin{tikzpicture}
\pgftransformcm{0.55}{0}{0}{0.37}{\pgfpoint{0cm}{0cm}}
\foreach \x in {-6,-5,...,7}{
      \foreach \y in {-7,-6,...,8}{
        \node[draw,circle,inner sep=0.5pt,fill] at (\x,\y) {}; } }
\draw  (-6,2) -- (-5,1) ;
\draw  (-6,1) -- (-5,2) ;
\draw  [dashed](-6,2) -- (-5,2) ;
\draw  (-6,-1) -- (-5,0) ;
\draw  (-6,0) -- (-5,-1) ;
\draw  [dotted](-6,0) -- (-5,0) ;
\foreach \y in {-6,-7,-5,-4,-3,-2,3,4,5,6,7,8}
{\draw (-6,\y) -- (-5,\y);}

\draw  (6,2) -- (7,1) ;
\draw  (6,1) -- (7,2) ;
\draw  [dashed](6,2) -- (7,2) ;
\draw  (6,-1) -- (7,0) ;
\draw  (6,0) -- (7,-1) ;
\draw  [dotted](6,0) -- (7,0) ;
\foreach \y in {-6,-7,-5,-4,-3,-2,3,4,5,6,7,8}
{\draw (6,\y) -- (7,\y);}

\draw  (-5,3) -- (-4,2) ;
\draw  (-5,2) -- (-4,3) ;
\draw  [dashed](-5,3) -- (-4,3) ;
\draw  (-5,-1) -- (-4,-2) ;
\draw  (-5,-2) -- (-4,-1) ;
\draw  [dotted](-5,-1) -- (-4,-1) ;
\foreach \y in {-6,-7,-5,-4,-3,0,1,4,5,6,7,8}
{\draw (-5,\y) -- (-4,\y);}

\draw  (5,3) -- (6,2) ;
\draw  (5,2) -- (6,3) ;
\draw  [dashed](5,3) -- (6,3) ;
\draw  (5,-1) -- (6,-2) ;
\draw  (5,-2) -- (6,-1) ;
\draw  [dotted](5,-1) -- (6,-1) ;
\foreach \y in {-6,-7,-5,-4,-3,0,1,4,5,6,7,8}
{\draw (5,\y) -- (6,\y);}

\draw  (-4,4) -- (-3,3) ;
\draw  (-4,3) -- (-3,4) ;
\draw  [dashed](-4,4) -- (-3,4) ;
\draw  (-4,-2) -- (-3,-3) ;
\draw  (-4,-3) -- (-3,-2) ;
\draw  [dotted](-4,-2) -- (-3,-2) ;
\foreach \y in {-6,-7,-5,-4,-1,0,1,2,5,6,7,8}
{\draw (-4,\y) -- (-3,\y);}

\draw  (4,4) -- (5,3) ;
\draw  (4,3) -- (5,4) ;
\draw  [dashed](4,4) -- (5,4) ;
\draw  (4,-2) -- (5,-3) ;
\draw  (4,-3) -- (5,-2) ;
\draw  [dotted](4,-2) -- (5,-2) ;
\foreach \y in {-6,-7,-5,-4,-1,0,1,2,5,6,7,8}
{\draw (4,\y) -- (5,\y);}

\draw  (-3,5) -- (-2,4) ;
\draw  (-3,4) -- (-2,5) ;
\draw  [dashed](-3,5) -- (-2,5) ;
\draw  (-3,-3) -- (-2,-4) ;
\draw  (-3,-4) -- (-2,-3) ;
\draw  [dotted](-3,-3) -- (-2,-3) ;
\foreach \y in {-6,-7,-5,-2,-1,0,1,2,3,6,7,8}
{\draw (-3,\y) -- (-2,\y);}

\draw  (3,5) -- (4,4) ;
\draw  (3,4) -- (4,5) ;
\draw  [dashed](3,5) -- (4,5) ;
\draw  (3,-3) -- (4,-4) ;
\draw  (3,-4) -- (4,-3) ;
\draw  [dotted](3,-3) -- (4,-3) ;
\foreach \y in {-6,-7,-5,-2,-1,0,1,2,3,6,7,8}
{\draw (3,\y) -- (4,\y);}

\draw  (-2,6) -- (-1,5) ;
\draw  (-2,5) -- (-1,6) ;
\draw  [dashed](-2,6) -- (-1,6) ;
\draw  (-2,-4) -- (-1,-5) ;
\draw  (-2,-5) -- (-1,-4) ;
\draw  [dotted](-2,-4) -- (-1,-4) ;
\foreach \y in {-6,-7,-3,-2,-1,0,1,2,3,4,7,8}
{\draw (-2,\y) -- (-1,\y);}

\draw  (2,6) -- (3,5) ;
\draw  (2,5) -- (3,6) ;
\draw  [dashed](2,6) -- (3,6) ;
\draw  (2,-4) -- (3,-5) ;
\draw  (2,-5) -- (3,-4) ;
\draw  [dotted](2,-4) -- (3,-4) ;
\foreach \y in {-6,-7,-3,-2,-1,0,1,2,3,4,7,8}
{\draw (2,\y) -- (3,\y);}

\draw  (-1,7) -- (0,6) ;
\draw  (-1,6) -- (0,7) ;
\draw  [dashed](-1,7) -- (0,7) ;
\draw  (-1,-5) -- (0,-6) ;
\draw  (-1,-6) -- (0,-5) ;
\draw  [dotted](-1,-5) -- (0,-5) ;
\foreach \y in {-7,-4,-3,-2,-1,0,1,2,3,4,5,8}
{\draw (-1,\y) -- (0,\y);}

\draw  (1,7) -- (2,6) ;
\draw  (1,6) -- (2,7) ;
\draw  [dashed](1,7) -- (2,7) ;
\draw  (1,-5) -- (2,-6) ;
\draw  (1,-6) -- (2,-5) ;
\draw  [dotted](1,-5) -- (2,-5) ;
\foreach \y in {-7,-4,-3,-2,-1,0,1,2,3,4,5,8}
{\draw (1,\y) -- (2,\y);}

\draw  (0,8) -- (1,7) ;
\draw  (0,7) -- (1,8) ;
\draw  [dashed](0,8) -- (1,8) ;
\draw  (0,-6) -- (1,-7) ;
\draw  (0,-7) -- (1,-6) ;
\draw  [dotted](0,-6) -- (1,-6) ;
\foreach \y in {-5,-4,-3,-2,-1,0,1,2,3,4,5,6}
{\draw (0,\y) -- (1,\y);}
\end{tikzpicture}
\caption{$\mathcal{D}_Q,  Q=\lbrack n \searrow 1\nearrow n\rbrack$}
 \end{figure}
\begin{proof}
The argument is exactly the same as in the previous case for
$1\le \alpha\le n$. When $\alpha=0$, since there is exactly one possible
matching of $\{1,2,\cdots n\}$ with $\{2',3',\cdots, (n+1)'\}$ and all edges in this matching
have weight 1, the ideal
of $m_Q^{-1}(\overline{B^{GL}_-\widetilde{r_0} B^{GL}})$
in this case is generated by 1.
Figure 3 shows an illustration for this 
case.
\end{proof}
 \begin{prop}
 On $X_\circ^w$ where $w=\prod Q$, $Q=[i\searrow 1\nearrow(n-1)0n\searrow (i+1)]$,
 $i\le n-1$, the defining ideal for $X_{r_\alpha}\cap X_\circ^w$ is generated by
 \[
\begin{cases}
z_i & \text{if }\alpha = 1\\
z_i+z_{i-1}z_{i+1}+\cdots +z_{i-(\alpha-1)}z_{i+(\alpha-1)} & \text{if }2\le \alpha \le i \\
\sum_{k=1}^{n+1-\alpha} z_{n+i-k}z_{(n+i)+(k-1)} &\text{if } i+1\le \alpha \le n \\
z_{n+i-1} &\text{if }\alpha = 0
\end{cases} 
 \]
 \end{prop}
 
 \begin{proof}

 Let $M_Q=m_Q(z_1,\cdots, z_{2n-1})$. 
 Figure 4 shows an example of the
 diagram.
 By the same argument on the diagram
 as in proposition 1, we notice that for $1\le \alpha \le n$ the
 northwest $\alpha\times\alpha$ minor and the northwest  
 $(2(n+1)-\alpha)\times(2(n+1)-\alpha)$ minor 
 in $M_Q$ are the same up to sign,
 so to compute pullbacks of the Fulton generators
 we only need to consider the 
 $\alpha\times\alpha$ minors. 

 \begin{figure}[h]
     \centering
     
\begin{tikzpicture}
\pgftransformcm{0.55}{0}{0}{0.37}{\pgfpoint{0cm}{0cm}}
\tikzstyle{selected edge} = [draw,line width=2pt,-,red!50]

\foreach \x in {-6,-5,...,7}{
      \foreach \y in {-7,-6,...,8}{
        \node[draw,circle,inner sep=0.5pt,fill] at (\x,\y) {}; } }
\draw[selected edge]  (-6,5)node[left]{\scriptsize $i+1$} -- (-3,8)--(7,8) ;

\draw[selected edge]  (-6,7) -- (-5,7)--(-4,6)--(-3,6)--(-2,7)--(7,7) ;
\draw[selected edge]  (-6,6) node[left]{\scriptsize $i$}-- (-5,5)--(-2,5)--(-1,6)--(7,6);
\draw[selected edge]  (-6,4)--(-1,4)--(0,5)--(6,5)--(7,4);
\draw  (-6,5) -- (-5,6) ;
\draw  (-6,6) -- (-5,5) ;
\draw  [dashed](-6,6) -- (-5,6) ;
\draw  (-6,-4) -- (-5,-5) ;
\draw  (-6,-5) -- (-5,-4) ;
\draw  [dotted](-6,-4) -- (-5,-4) ;
\foreach \y in {-6,-7,-3,-2,-1,0,1,2,3,4,7,8}
{\draw (-6,\y) -- (-5,\y);}

\draw  (-5,6) -- (-4,7) ;
\draw  (-5,7) -- (-4,6) ;
\draw  [dashed](-5,7) -- (-4,7) ;
\draw  (-5,-5) -- (-4,-6) ;
\draw  (-5,-6) -- (-4,-5) ;
\draw  [dotted](-5,-5) -- (-4,-5) ;
\foreach \y in {-7,-4,-3,-2,-1,0,1,2,3,4,5,8}
{\draw (-5,\y) -- (-4,\y);}

\draw  (-4,7) -- (-3,8) ;
\draw  (-4,8) -- (-3,7) ;
\draw  [dashed](-5,8) -- (-3,8) ;
\draw  (-4,-6) -- (-3,-7) ;
\draw  (-4,-7) -- (-3,-6) ;
\draw  [dotted](-4,-6) -- (-3,-6) ;
\foreach \y in {-5,-4,-3,-2,-1,0,1,2,3,4,5,6}
{\draw (-4,\y) -- (-3,\y);}

\draw  (-3,6) -- (-2,7) ;
\draw  (-3,7) -- (-2,6) ;
\draw  [dashed](-3,7) -- (-2,7) ;
\draw  (-3,-5) -- (-2,-6) ;
\draw  (-3,-6) -- (-2,-5) ;
\draw  [dotted](-3,-5) -- (-2,-5) ;
\foreach \y in {-7,-4,-3,-2,-1,0,1,2,3,4,5,8}
{\draw (-3,\y) -- (-2,\y);}

\draw  (-2,5) -- (-1,6) ;
\draw  (-2,6) -- (-1,5) ;
\draw  [dashed](-2,6) -- (-1,6) ;
\draw  (-2,-4) -- (-1,-5) ;
\draw  (-2,-5) -- (-1,-4) ;
\draw  [dotted](-2,-4) -- (-1,-4) ;
\foreach \y in {-6,-7,-3,-2,-1,0,1,2,3,4,7,8}
{\draw (-2,\y) -- (-1,\y);}

\draw  (-1,4) -- (0,5) ;
\draw  (-1,5) -- (0,4) ;
\draw  [dashed](-1,5) -- (0,5) ;
\draw  (-1,-3) -- (0,-4) ;
\draw  (-1,-4) -- (0,-3) ;
\draw  [dotted](-1,-3) -- (0,-3) ;
\foreach \y in {-6,-7,-5,-2,-1,0,1,2,3,6,7,8}
{\draw (-1,\y) -- (0,\y);}

\draw  (0,3) -- (1,4) ;
\draw  (0,4) -- (1,3) ;
\draw  [dashed](0,4) -- (1,4) ;
\draw  (0,-2) -- (1,-3) ;
\draw  (0,-3) -- (1,-2) ;
\draw  [dotted](0,-2) -- (1,-2) ;
\foreach \y in {-6,-7,-5,-4,-1,0,1,2,5,6,7,8}
{\draw (0,\y) -- (1,\y);}

\draw  (1,2) -- (2,3) ;
\draw  (1,3) -- (2,2) ;
\draw  [dashed](1,3) -- (2,3) ;
\draw  (1,-1) -- (2,-2) ;
\draw  (1,-2) -- (2,-1) ;
\draw  [dotted](1,-1) -- (2,-1) ;
\foreach \y in {-6,-7,-5,-4,-3,0,1,4,5,6,7,8}
{\draw (1,\y) -- (2,\y);}

\draw  (2,2) -- (3,0) ;
\draw  (2,0) -- (3,2) ;
\draw  [dashed](2,2) -- (3,2) ;
\draw  (2,-1) -- (3,1) ;
\draw  (2,1) -- (3,-1) ;
\draw  [dotted](2,1) -- (3,1) ;
\foreach \y in {-6,-7,-5,-4,-3,-2,3,4,5,6,7,8}
{\draw (2,\y) -- (3,\y);}

\draw  (3,2) -- (4,1) ;
\draw  (3,1) -- (4,2) ;
\draw  [dashed](3,2) -- (4,2) ;
\draw  (3,-1) -- (4,0) ;
\draw  (3,0) -- (4,-1) ;
\draw  [dotted](3,0) -- (4,0) ;
\foreach \y in {-6,-7,-5,-4,-3,-2,3,4,5,6,7,8}
{\draw (3,\y) -- (4,\y);}

\draw  (4,2) -- (5,3) ;
\draw  (4,3) -- (5,2) ;
\draw  [dashed](4,3) -- (5,3) ;
\draw  (4,-1) -- (5,-2) ;
\draw  (4,-2) -- (5,-1) ;
\draw  [dotted](4,-1) -- (5,-1) ;
\foreach \y in {-6,-7,-5,-4,-3,0,1,4,5,6,7,8}
{\draw (4,\y) -- (5,\y);}

\draw  (5,3) -- (6,4) ;
\draw  (5,4) -- (6,3) ;
\draw  [dashed](5,4) -- (6,4) ;
\draw  (5,-2) -- (6,-3) ;
\draw  (5,-3) -- (6,-2) ;
\draw  [dotted](5,-2) -- (6,-2) ;
\foreach \y in {-6,-7,-5,-4,-1,0,1,2,5,6,7,8}
{\draw (5,\y) -- (6,\y);}

\draw  (6,4) -- (7,5) ;
\draw  (6,5) -- (7,4) ;
\draw  [dashed](6,5) -- (7,5) ;
\draw  (6,-3) -- (7,-4) ;
\draw  (6,-4) -- (7,-3) ;
\draw  [dotted](6,-3) -- (7,-3) ;
\foreach \y in {-6,-7,-5,-2,-1,0,1,2,3,6,7,8}
{\draw (6,\y) -- (7,\y);}
\end{tikzpicture}
\caption{$\mathcal{D}_Q, Q=[i\searrow 1\nearrow(n-1)0n\searrow (i+1)]$}
\end{figure}
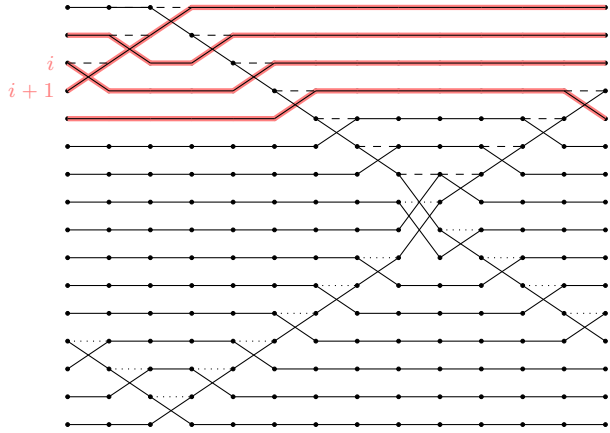

 For $1\le \alpha \le i$ the part of the diagram that is relevant has
 the same structure as in Proposition 1, and the $\alpha\times\alpha$ minor 
 is $-z_i$ when $\alpha=1$ and 
 $-z_i-z_{i-1}z_{i+1}-\cdots -z_{i-(\alpha-1)}z_{i+(\alpha-1)} $ when
 $2\le \alpha \le i$.

 Now suppose $i+1\le \alpha \le n$. Notice that $i+1\leftrightarrow 1'$ is 
 forced. For $1\le j\le i$, there is one path to $j'$ and another path to $1'$.
  Therefore to proceed legally we have $j\leftrightarrow j'$ for $1\le j\le i$.
Now for all $j$ where  $i+1<j\le \alpha$, 
 we also have $j\lra j'$, since $j$ can be matched with $j'$ or $(i+1)'$,
 but in the latter case $j'$ cannot be matched.
 The colored edges in Figure 4 illustrates
 an example of a forced configuration.
 The only pair left
 is $1$ and $(i+1)'$, and there are $n+1-\alpha$ such paths, each 
 picking up a term $z_{n+i-k}z_{(n+i)+(k-1)}$ for some $1\le k\le n+1-\alpha$.
 Notice also from the diagram that the number of crossings we encounter
 is always odd. Therefore, the $\alpha\times\alpha$ minor is
 $-\sum_{k=1}^{n+1-\alpha} z_{n+i-k}z_{(n+i)+(k-1)} $.  

Finally, when $\alpha=0$, we claim that the 
ideal generated by pullbacks of the Fulton
generators is generated by $z_{n+i-1}$.
First observe that all $n\times n$ minors
in the northwest $(n+1)\times (n+1)$ blocks
are divisible by $z_{n+i-1}$. To see this,
for each minor
 consider the $(n+i-1)$th column
that corresponds to $r_0$. If neither of 
the edge whose weight is $\pm z_{n+i-1}$ is
used, then only $n-1$ edges in that column
are connected to both
$\{1,\cdots ,n+1\}$ and $\{1,\cdots,(n+1')\}$.
Therefore one of these (dotted/dashed) edges must be used. Observe also that the minor that is
computed by
following $\{2,\cdots,n+1\}$
to $\{1',\cdots,n'\}$ is exactly $z_{n+i-1}$.
The claim then follows.
  \end{proof}
 On $\overline{PSO(2n)/SO(2n-1)}\cong \mathbb{P}^{2n-1}$, following the coordinatization convention 
on $\mathbb{P}^{2n-1}$
set in Section 1,
coordinatize each 
 standard affine chart $U_l$ so that
$U_l\hookrightarrow\mathbb{P}^{2n-1}$,
$(u_1,u_2,\cdots, \hat{u_l},\cdots, u_{2n})\mapsto
[u_1:u_3:\cdots :1:\cdots:u_{2n-2}: u_{2n}]$
where $\hat{u_l}$
means omitting the coordinate $u_l$.

\begin{prop}
The isomorphisms  $c_l:U_l\to X_\circ^{w_l}$ explicitly given as follows
preserve the stratification on $\mathbb{P}^{2n-1}$
defined in Section 1 and the stratification
of the corresponding opposite Bruhat cells
stratified by Schubert varieties
:
$$c_1(u_2,\cdots, u_{2n})=(u_3,u_5,\cdots,u_{2n-1},u_2+u_3u_4+\cdots +u_{2n-1}u_{2n},-u_{2n},-u_{2n-2},\cdots, -u_4)$$
 (or, $c_1(u_2,\cdots, u_{2n})=(z_1,\cdots,z_{2n-1})$,
 where 
 $z_n=u_2+\sum_{i=1}^{n-1}u_{2i+1}u_{2i+2}$,
 $z_j=u_{2j+1}$ for $1\le j\le n-1$, and $z_j=-u_{2(2n+1-j)}$),
 \[c_2(u_1,u_3\cdots, u_{2n})=(u_3,u_5,\cdots,u_{2n-1},u_1+u_3u_4+\cdots +u_{2n-1}u_{2n},-u_{2n},-u_{2n-2},\cdots, -u_4),\]
 and for $2\le i\le n$,
 \vspace{-.3cm}
 \begin{align*}
 c_{2i-1}(u_1,\cdots,\widehat{u_{2i-1}},\cdots, u_{2n})=(u_{2i+1},u_{2i+3},\cdots , u_{2n-1}, u_{2i} 
 +\sum_{j=1,j\neq i}^nu_{2j-1}u_{2j},\\ -u_{2n},-u_{2n-2},
 \cdots, -u_{2i+2}, u_{2i-3}, u_{2i-5},\cdots, u_1, u_2, u_4, \cdots, u_{2i-2}  )
 \end{align*}
(or, 
 $c_{2i-1}(u_1,\cdots,\widehat{u_{2i-1}},\cdots, u_{2n} )=(z_1,\cdots, z_{2n-1})$ where \[z_j=\begin{cases}
 u_{2i+2j-1} & \text{ if } 1\le j \le n-i \\
 u_{2i} +\sum_{j=1,j\neq i}^nu_{2j-1}u_{2j} & \text{ if } j=n-i+1 \\
 -u_{4n-2i+4-2j} & \text{ if } n-i+2\le j \le 2n-2i+1 \\
 u_{4n-2i+1-2j} & \text{ if } 2n-2i+2 \le j \le 2n-i \\
 u_{2(j-2n+i)} &\text{ if } 2n-i+1\le j\le 2n-1),
 \end{cases}\]

 \begin{align*}
 c_{2i}(u_1,\cdots,\widehat{u_{2i-1}},\cdots, u_{2n})=(u_{2i-2},u_{2i-4},\cdots , u_{2}, u_{1},u_3,\cdots u_{2i-3},-u_{2i+2},-u_{2i+4},\cdots ,\\ -u_{2n}, 
u_{2i-1} +\sum_{j=1,j\neq i}^nu_{2j-1}u_{2j}, u_{2n-1},u_{2n-3},
 \cdots, u_{2i+1} ).
 \end{align*}
\end{prop} 
Notice that for $2\le i\le n$, the chart map
 $c_{2i}$ is the map $c_{2i-1}$ with 
 $u_{2i}$ replaced with $u_{2i-1}$ and
 all coordinates reversed.
\begin{proof}
  We may now compute for each
  $w_l$ $(1\le l\le 2n)$ the pullback of the 
  generator of the principal ideal
  $X_{r_\alpha}\cap X_\circ^{w_l}$ via the chart map $c_l$
  for all $0\le \alpha\le n$.
  
  For $w_1=\prod Q_1$, $Q_1=[0(n-1)\searrow 1\nearrow (n-1)0]$,
  The relevant polynomials to pull back are
  those in Proposition 2. $z_n$ pulls back
  to $u_2+u_3u_4+\cdots u_{2n-1}u_{2n}$
  on $U_1$. When $2\le \alpha \le n$,
  $z_n+z_{n-1}z_{n+1}+\cdots z_{n-(\alpha-1)}z_{n+(\alpha-1)}$ 
  pulls back to $u_2+\sum_{i=1}^{n-\alpha} u_{2i+1}u_{2i+2}$. The polynomial for
  $\alpha=0$ is the same as that for $\alpha=n$.
  The case for $w_2$ is similar; 
  the polynomials obtained via pullbacks
  are the same except with all
  $u_2$ replaced by $u_1$.
  
    For $2\le i\le n$,
    $w_{2i-1}=\prod Q_{2i-1}$, $Q_{2i-1}=[(n+1-i)\searrow 1\nearrow (n-1)0n\searrow (n+2-i)]$,
    the relevant polynomials are as follows:
 \[
\begin{cases}
z_{n+1-i} & \text{if }\alpha = 1\\
z_{n+1-i}+z_{n-i}z_{n-i+2}+\cdots +z_{n-i-\alpha+2}z_{n-i+\alpha} & \text{if }2\le \alpha \le n-i+1 \\
\sum_{k=1}^{n+1-\alpha} z_{2n+1-i-k}z_{2n-i+k} &\text{if } n-i+2\le \alpha \le n \\
z_{2n-i} &\text{if }\alpha = 0
\end{cases} 
 \]
 $z_{n+1-i}$ pulls back to $ u_{2i} +\sum_{j=1,j\neq i}^n u_{2j-1}u_{2j} $. For $2\le \alpha \le n-i+1$,
 $z_{n+1-i}+z_{n-i}z_{n-i+2}+\cdots +z_{n-i-\alpha+2}z_{n-i+\alpha}$ pulls back to
 $ u_{2i} +\sum_{j=1,j\neq i}^{n-\alpha+1} u_{2j-1}u_{2j}$.
 For $n-i+2\le \alpha \le n$, $\sum_{k=1}^{n+1-\alpha} z_{2n+1-i-k}z_{2n-i+k}$ pulls back to
 $\sum_{j=1}^{n-\alpha+1} u_{2j-1}u_{2j}$. For 
 $\alpha=0$, $z_{2n-i}$ pulls back to $u_1$. 
 
For $w_{2i}=\prod Q_{2i}$, since $Q_{2i}$ is exactly the reverse
of $Q_{2i-1}$, and the chart map $c_{2i}$
is exactly $c_{2i-1}$ with the coordinates
reversed and $u_{2i}$ swapped for $u_{2i-1}$,
the polynomials pulled back are exactly the same
as those for $w_{2i-1}$ but with  $u_{2i}$ swapped for $u_{2i-1}$.
This proposition establishes our main theorem, Theorem 2
stated in Section 1.
\end{proof}

   \bibliographystyle{alpha}
\bibliography{ref}

 \end{document}